\documentclass [11pt,reqno] {amsart}

\usepackage{latexsym}
\begin{document}
\hyphenation{na tu ral}
\newtheorem{Thm}{Theorem }
\newtheorem{Pro}{Proposition }
\newtheorem{Cor}{Corollary }
\newtheorem{Lem}{Lemma}
\newtheorem{Rem}{Remark}
\newtheorem{Def}{Definition }
\newtheorem{Note}{Note }
\newtheorem{Exam}{Examples}
\newcommand{\espacio} [1]{$\Lambda_{-#1, #1-1}(z)$}
\newcommand{\aprox} [1]{$HF_{-#1, #1-1}(z)$}
\newcommand{\aproxh} [1]{$H_{-#1, #1-1}(z)$}
\title[  Nodal systems and functions]{ About nodal systems  for   Lagrange  interpolation    on the circle }
\author{ E. Berriochoa, A. Cachafeiro, J.M. Garc\'{\i}a Amor}
\address {{\ El\'{\i}as Berriochoa}\\
Departamento de Matem\'atica Aplicada I\\
Facultad de Ciencias\\
Universidad de Vigo\\
Ourense, Spain} \email{esnaola@uvigo.es}
\address {{\ Alicia Cachafeiro}\\
Departamento de Matem\'atica Aplicada I\\
Escuela de Ingenier\'{\i}a  Industrial\\
Universidad de Vigo\\
36310 Vigo, Spain} \email{acachafe@uvigo.es}
\address {{\ Jos\'e M. Garc\'{\i}a Amor}\\
Departamento de Matem\'atica Aplicada I\\
Escuela de Ingenier\'{\i}a  Industrial\\
Universidad de Vigo\\
36310 Vigo, Spain} \email{garciaamor@uvigo.es}
\vspace{1cm}
\input {amssym.def}
\input {amssym}

\thanks{ 
The research was  supported by Ministerio de Educaci\'on y  Ciencia   under grant number MTM2008-00341.}


\begin{abstract}
We study the convergence of the Laurent polynomials of Lagrange  interpolation  on the unit circle for  continuous functions satisfying a condition about their modulus of continuity.
The novelty of the result is that now the nodal systems are   more general than those constituted by the $n$ roots of complex  unimodular numbers and the class of functions is wider than the usually studied.
Moreover,  some  consequences for the Lagrange interpolation on $[-1,1]$ and the Lagrange trigonometric interpolation are obtained.
\end{abstract}
\maketitle
\vspace{.6cm}

\section{Introduction} The aim of this paper is to study the Lagrange interpolation problem on the unit circle $\mathbb{T}:=\{ z: \, \vert z \vert =1\}$ for nodal systems more general than those constituted by the $n$ roots of complex unimodular numbers.  This last case has been studied in \cite{DG2}, where there is posed as an open problem its extension to more general nodal systems. Recently a similar problem has been solved in \cite{BeCaM} for the Hermite interpolation problem.  Now we follow the ideas in \cite{BeCaM} to obtain some results for the Lagrange case.
Moreover, in  \cite{DG2}  it is obtained a result about convergence of the interpolants for continuous functions satisfying a condition related with their modulus of continuity. In the present paper our aim is to obtain a similar result for the new nodal systems and with a weaker condition on the modulus of continuity for the functions.\\

The Lagrange interpolation problem  on the real line has been widely studied for a long time and many results about convergence are known, (see \cite{BeTr}, \cite{Rv}, \cite{SzV} and \cite{Tu}). If we only assume the continuity of the function, it is well known that the behavior is rather irregular. Faber has proved that for each nodal system there exists a continuous function such that the sequence of Lagrange interpolation polynomials is not uniformly convergent. Bernstein has also proved the existence of a continuous function such that the sequence of Lagrange interpolation polynomials is unbounded on a prefixed point.  In the case of the nodal systems constituted by the zeros of the Tchebychef polynomials of the first kind, many results are known. Although these last nodal systems are good for interpolation, Gr\"{u}nwald  in \cite{Gru} and Marcinkiewicz in \cite{Mar} have proved the existence of a continuous function such that the sequence of Lagrange interpolation polynomials, corresponding to the Tchebychef nodal system, is divergent. After this result a natural problem was to obtain an analogous result for an arbitrary nodal system. This result was obtained by Erd\"{o}s and V\'ertesi in \cite{EV}, where they prove  that for each nodal system on $[-1,1]$ there exists a continuous function such that the sequence of Lagrange  interpolation polynomials diverges for almost every point in $[-1,1]$.
Thus, to obtain better properties about the convergence of the sequence of Lagrange interpolation polynomials, it is needed to impose some restriction to the function, such as, a condition on its modulus of continuity. In the case of Jacobi abscissas, Szeg\H{o} has obtained important results about convergence by imposing some conditions to the modulus of continuity of the function, (see \cite{Sze}). For example, in the case of the Tchebychef abscissas of first kind, he obtained the uniform convergence  to the function on $[-1,1]$, under the assumption that its modulus of continuity is $o(\vert \log \delta \vert^{-1})$. Szeg\H{o} has also obtained uniform convergence  of the sequence of Lagrange interpolation polynomials for more general nodal systems, under the assumptions that the nodes are the zeros of the orthogonal polynomials with respect to a weight function $w(x)$ such that $w(x)\sqrt{1-x^2} \geq \mu >0 ,\; x \in (-1,1)$ and the modulus of continuity of the functions is $o(\delta^{\frac{1}{2}})$ with
$ \delta \rightarrow 0.$\\

In the present paper we improve some results about convergence of the Lagrange interpolation polynomials in $[-1,1]$, by using the Szeg\H{o} transformation and   the results concerning the unit circle.
The organization of the paper is the following.
In section $2$ we obtain our main result concerning the uniform convergence of the Laurent polynomial of Lagrange interpolation for nodal systems described in terms of some properties and  for  continuous functions with modulus of continuity $o(\delta^p)$ when $\delta\rightarrow 0$ and $p \geq \frac{1}{2}.$
Section $3$ is devoted to obtain some consequences of the preceding results concerning the Lagrange interpolation on $[-1,1]$.
Finally, in the last section,  we obtain some improvements concerning  the  Lagrange trigonometric interpolation.\\

\section{Lagrange  Interpolation in the space of Laurent polynomials}
Let  $\{z_j\}_{j=1}^{n}$ be a set of complex numbers such that $\vert z_j\vert=1$ for all $j=1,\cdots,n$ and  $z_i \neq z_j$ for $i \neq j$. Let     $\{u_j\}_{j=1}^{n}$   be a set of  arbitrary complex numbers,  and let $p(n)$ and $q(n)$ be two nondecreasing  sequences of nonnegative integers such that  $p(n)+q(n)=n-1,\, n \geq 2$ with $\lim_{n \rightarrow \infty} p(n)= \lim_{n \rightarrow \infty} q(n)= \infty$.

We recall that the Lagrange interpolation problem in the space of Laurent polynomials  consists in determining the unique Laurent polynomial $L_{-p(n), q(n)}(z) \in \Lambda_{-p(n), q(n)}= span\{z^k: -p(n) \leq k \leq q(n) \}$ such that  :
\begin{eqnarray}
L_{-p(n), q(n)}(z_j)=u_j,\;\mbox{for}\; j=1,\cdots, n.
\label{Eq:2}\end{eqnarray}

If we denote by $W_n(z)= \displaystyle \prod_{j=1}^n(z-z_j)$ the nodal polynomial, then $L_{-p(n), q(n)}(z)$ can be written as follows
\begin{eqnarray}
L_{-p(n), q(n)}(z)=\sum_{j=1}^n l_{j,n-1}(z)u_j,
\label{Eq:*}\end{eqnarray}
where $l_{j,n-1}(z)$ are the fundamental polynomials of Lagrange interpolation given by
\begin{eqnarray}
l_{j,n-1}(z)=\frac{z_j^{p(n)} W_n(z)}{W_n'(z_j)(z-z_j)z^{p(n)}},\; \mbox{for}\; j=1,\cdots,n,
\label{Eq:**}\end{eqnarray}
and they are characterized by satisfying $l_{j,n-1}(z_k)=\delta_{j,k},\;\forall j,k.$

We are also going to consider the Lagrange interpolation polynomial for a function $F$ defined on $\mathbb{T}$, that we are going to denote by
$L_{-p(n), q(n)}(F;z)$ and which is characterized by fulfilling the conditions  $L_{-p(n), q(n)}(F;z_j)=F(z_j)$ for $j=1, \cdots,n.$

When the nodal system is constituted by the $n$-roots  of a complex number with modulus $1$, and the function $F$ is continuous on $ \mathbb{T}$ and its modulus of continuity satisfies
$ \lambda(F,\delta)= \mathcal{O}(\delta^p), \; p > \frac{1}{2}$, the following result about convergence is known, (see \cite{DG2}).

\begin{Thm} Let $F$ be a continuous function on $\mathbb{T}$, let $p(n)$ and $q(n)$ be two nondecreasing  sequences of nonnegative integers such that  $p(n)+q(n)=n-1$ and  $\displaystyle \lim_{n \rightarrow \infty}\frac{p(n)}{n-1}=r$ with $0<r <1$, and assume that the modulus of continuity of $F$, $ \lambda(F,\delta)= \mathcal{O}(\delta^p)$ for some $p > \frac{1}{2}$,
if $\delta \rightarrow 0$.

Let $L_{-p(n), q(n)}(F;z)$ be the Laurent polynomial of Lagrange interpolation for the function $F$ with nodal system $\{z_j\}_{j=1}^{n}$ the $n$-roots of complex numbers $\tau_n$ with $ \vert \tau_n \vert=1.$

 Then $\displaystyle \lim_{n \rightarrow \infty}L_{-p(n), q(n)}(F;z)=F(z)$ uniformly on $\mathbb{T}$.
\label{TL:1}\end{Thm}
\begin{proof} See \cite{DG2}.
\end{proof}
The main tools to prove the preceding result are the explicit expression of the Laurent polynomial of Lagrange interpolation and some properties concerning the nodal system.
In \cite{BeCaM}  the Hermite interpolation  problem was studied for general nodal systems satisfying certain properties. Following  similar ideas we prove, in the next theorem,  a result about the convergence of the Lagrange interpolants for a wider class of functions and more general nodal systems.

\begin{Thm} Let $F$ be a continuous function on $\mathbb{T}$, with modulus of continuity  $ \lambda(F,\delta)= o(\delta^{\frac{1}{2}})$, if $\delta \rightarrow 0$. Let $p(n)$ and $q(n)$ be two nondecreasing  sequences of nonnegative integers such that  $p(n)+q(n)=n-1$ and  $\displaystyle \lim_{n \rightarrow \infty}\frac{p(n)}{n-1}=r$ with $0<r <1$.

Let $\{z_j\}_{j=1}^n$ be a set of complex numbers such that $\vert z_j \vert=1$ for all $j=1,\cdots,n$ and $ z_i \neq z_j$ for $i \neq j$ and let
$W_n(z)=\displaystyle \Pi_{j=1}^n (z-z_j)$ be the nodal polynomial.  Assume that  there exist  positive constants $B$ and $L$ such that for every $z \in \mathbb{T}$ and $n$ large enough the following relations hold:
\begin{enumerate}
  \item [(i)]  $B \leq \displaystyle \frac{\vert W_n'(z) \vert}{n},$
  \item [(ii)] $\displaystyle \frac{\vert W_n(z) \vert^2}{n^2} \sum_{j=1}^n \displaystyle \frac{1}{\vert z-z_j \vert^2}\leq L.$
\end{enumerate}
If $L_{-p(n), q(n)}(F;z) \in \Lambda_{-p(n),q(n)}$ is the Laurent polynomial of Lagrange interpolation related to the nodal system and the function $F$,
  then $\displaystyle \lim_{n \rightarrow \infty}L_{-p(n), q(n)}(F;z)=F(z)$ uniformly on $\mathbb{T}$.
\label{TL:2}\end{Thm}
\begin{proof} First we prove that there exists a positive constant $C$ such that
$\sum_{j=1}^n \vert l_{j,n-1}(z)\vert \leq C \sqrt{n}$ for every $z \in \mathbb{T}$ and $n$ large enough. Indeed, taking into account (\ref{Eq:**}) and applying the hypothesis we get:
\begin{equation*}\begin{split} \sum_{j=1}^n \vert l_{j,n-1}(z)\vert= \sum_{j=1}^n \frac{\vert W_n(z) z_j^{p(n)} \vert }{\vert W'_n(z_j) (z-z_j) z^{p(n)} \vert}=\sum_{j=1}^n \frac{\vert W_n(z) \vert }{\vert W'_n(z_j) (z-z_j)  \vert}\leq
\frac{1}{B n} \sum_{j=1}^n \frac{\vert W_n(z) \vert }{\vert z-z_j  \vert} \leq \\\frac{1}{B n}\left( \sum_{j=1}^n \frac{\vert W_n(z) \vert^2 }{\vert z-z_j  \vert^2}\right)^{\frac{1}{2}}
 ( \sum_{j=1}^n 1)^{\frac{1}{2}}\leq
 \frac{\sqrt{L}}{B}\sqrt{n}.
 \end{split}\end{equation*}
Let us consider the Laurent polynomial of best uniform approximation to $F$, $T_{-p(n),q(n)}(z) \in \Lambda_{-p(n),q(n)}$.  If $E_{-p(n),q(n)}(F)= \max_{z \in \mathbb{T}}\vert  F(z) - T_{-p(n),q(n)}(z)\vert,$
then it holds that  $$E_{-p(n),q(n)}(F) \leq 2 \lambda(F,\frac{\pi}{s(n)}),$$ where $s(n)= \min(p(n),q(n))$, (see \cite{DG2}).
Since $\displaystyle \lim_{n \rightarrow \infty} \frac{\pi}{s(n)}=0$, then by hypothesis $\lambda(F,\frac{\pi}{s(n)})= o((\frac{\pi}{s(n)})^{\frac{1}{2}})$.\\
If we write
\begin{equation*}\begin{split}
    F(z)-L_{-p(n),q(n)}(F;z)=F(z)-T_{-p(n),q(n)}(z)-L_{-p(n),q(n)}(F;z)+T_{-p(n),q(n)}(z)=\\
    F(z)-T_{-p(n),q(n)}(z)-L_{-p(n),q(n)}(F-T_{-p(n),q(n)};z)=\\
  F(z)-T_{-p(n),q(n)}(z)- \sum_{j=1}^n l_{j,n-1}(z)(F(z_j)-T_{-p(n),q(n)}(z_j)),
\end{split}\end{equation*}
then we have
\begin{equation*}\begin{split}
   \vert  F(z)-L_{-p(n),q(n)}(F;z)\vert \leq \vert F(z)-T_{-p(n),q(n)}(z)\vert + \sum_{j=1}^n \vert l_{j,n-1}(z)\vert  \vert F(z_j)-T_{-p(n),q(n)}(z_j)\vert \leq \\
   E_{-p(n),q(n)}(F)( 1+ \sum_{j=1}^n \vert l_{j,n-1}(z)\vert)\leq  2 \lambda(F,\frac{\pi}{s(n)})(1+ C \sqrt{n})= \\
   2\frac{ \lambda(F,\frac{\pi}{s(n)})}{\left(\frac{\pi}{s(n)}\right)^{\frac{1}{2}}}\frac{\sqrt{\pi}}{\left(\frac{s(n)}{n-1} \right)^{\frac{1}{2}}}\frac{1+C\sqrt{n}}{\sqrt{n-1}},
   \end{split}\end{equation*}
and it is easy to prove that the last expression tends to zero because
 $ \displaystyle \lim_{n \rightarrow \infty} \frac{ \lambda(F,\frac{\pi}{s(n)})}{\left(\frac{\pi}{s(n)}\right)^{\frac{1}{2}}}=0$
 and
\begin{equation*}\lim_{n \rightarrow \infty}\frac{s(n)}{n-1}=\frac{1}{2}\left( \lim_{n \rightarrow \infty}\frac{p(n)}{n-1} +  \lim_{n \rightarrow \infty} \frac{q(n)}{n-1}- \lim_{n \rightarrow \infty}  \frac{\vert p(n)-q(n) \vert }{n-1}\right)=
\left\{
  \begin{array}{ll}
    1-r, & \hbox{if}\; r \in (\frac{1}{2},1), \\
    r, & \hbox{if}\; r \in (0,\frac{1}{2}].
  \end{array}
\right.
\end{equation*}
\end{proof}
\begin{Rem}
\begin{enumerate}
\item [(i)] Since $\lambda(F,\delta)=o(\delta^p)$ for $p>\frac{1}{2}$ implies $\lambda(F,\delta)=o(\delta^{\frac{1}{2}})$, then the preceding result is also valid for functions with modulus of continuity $o(\delta^p)$, with $p > \frac{1}{2}$, if $\delta \rightarrow 0$. Hence, in the sequel and  for simplicity,  we establish all the results with the condition $\lambda(F,\delta)=o(\delta^{\frac{1}{2}})$.
\item [(ii)]  Since it is clear that the nodal systems in Theorem   \ref{TL:1} satisfy the hypothesis of  Theorem \ref{TL:2};   we have that the result given in Theorem \ref{TL:1} is also valid for functions with modulus of continuity $o(\delta^{\frac{1}{2}})$, if $\delta \rightarrow 0$.
\end{enumerate}
\end{Rem}
Next we recall a sufficient condition given in \cite{BeCaM} in order that the nodal system satisfy the conditions imposed in the previous theorem. We use the so called para-orthogonal polynomials, (see \cite{JNT}, \cite{Gol} and \cite{Sim})  and  the class of  measures satisfying the Szeg\H{o} condition, (see \cite{Sze}, \cite{Sim}, \cite{Nev} and \cite{Ger}). Notice that the nodal systems in Theorem \ref{TL:1} are constituted by the $n$ roots of complex unimodular numbers, and indeed they are the $n$ roots of the para-orthogonal polynomials with respect to the Lebesgue measure on $[0, 2 \pi]$.

\begin{Thm}
 Let $\nu$ be a measure on $[0,2\pi]$ in the  Szeg\H{o} class with Szeg\H{o} function having analytic extension up to $\vert z \vert >1$. Let $\{\phi_n(z)\}$ be
the MOPS($\nu$) and  $\omega_n(z,\tau)=\phi_n(z)+ \tau \phi^*_n(z)$, with $\vert \tau \vert=1$ the para-orthogonal polynomials. Then there exist  positive constants $A, B_1, B_2$ and $L$ such that for every $z \in \mathbb{T}$ and $n$ large enough the following relations hold:
\begin{enumerate}
  \item [(i)]  $\vert \omega_n(z,\tau) \vert \leq A,$
  \item [(ii)]  $B_1 \leq \displaystyle \frac{\vert \omega_n'(z,\tau) \vert}{n} \leq B_2,$
  \item [(iii)] $\displaystyle \frac{\vert \omega_n(z,\tau) \vert^2}{n^2} \sum_{j=1}^n \displaystyle \frac{1}{\vert z-z_j \vert^2}\leq L,$ where we assume that $z_1,\cdots,z_n$ are the zeros of $\omega_n(z,\tau).$
\end{enumerate}
\label{LemL:1}\end{Thm}
\begin{proof} See \cite{BeCaM}.
\end{proof}
Taking into account the preceding results, we are in conditions to prove the following corollary.

\begin{Cor}   Let $F$ be a continuous function on $\mathbb{T}$, with modulus of continuity $ \lambda(F,\delta)= o(\delta^{\frac{1}{2}})$, if $\delta \rightarrow 0$. Let $p(n)$ and $q(n)$ be two nondecreasing  sequences of nonnegative integers such that  $p(n)+q(n)=n-1$ and  $\displaystyle \lim_{n \rightarrow \infty}\frac{p(n)}{n-1}=r$ with $0<r <1$.

Let $\nu$ be a measure on $[0,2 \pi]$ in the Szeg\H{o} class with  Szeg\H{o} function having  analytic extension up to $\vert z \vert >1$. Let $\{\phi_n(z)\}$ be the MOPS($\nu$) and let $\omega_n(z,\tau)=\phi_n(z)+ \tau \phi^*_n(z)$, with $\vert \tau \vert=1$,  be the para-orthogonal polynomials.

If $L_{-p(n), q(n)}(F;z) \in \Lambda_{-p(n),q(n)}$ is the Laurent polynomial of Lagrange interpolation related to the function $F$ and  with nodal system  the zeros  of the para-orthogonal polynomials $\omega_n(z,\tau)$,
  then $\displaystyle \lim_{n \rightarrow \infty}L_{-p(n), q(n)}(F;z)=F(z)$ uniformly on $\mathbb{T}$.
\label{Cor:A1}\end{Cor}

\begin{proof} Taking into account that the zeros of $\omega_n(z,\tau)$ belong to $\mathbb{T}$, (see \cite{JNT}), the result is immediate from Theorems  \ref{TL:2} and \ref{LemL:1}.
\end{proof}
\begin{Rem} Notice that the preceding result is valid for the Bernstein-Szeg\H{o} measures, (see \cite{Sim}).
\end{Rem}

\section{Lagrange interpolation  on $[-1,1]$}
In this  section we present some consequences of  Theorem  \ref{TL:2} concerning the Lagrange interpolation problems on  $[-1,1]$. Let us recall that the Lagrange interpolation polynomial related to a nodal system $\{x_j\}_{j=1}^n \subset [-1,1]$ and satisfying the conditions $\{u_j\}_{j=1}^n$ is given by $l_{n-1}(x)=\displaystyle \sum_{j=1}^n \frac{p_n(x)}{p'_n(x_j)(x-x_j)}u_j,$ where $p_n(x)= \displaystyle \Pi_{j=1}^n(x-x_j)$.

\begin{Thm} Let $p_n(x)= \prod_{j=1}^n(x-x_j)$ be a nodal system in $[-1,1]$ such that $W_{2n}(z)=2^nz^n p_n(\frac{z+1/z}{2})$ satisfies  the following inequalities
$$ B \leq \frac{\vert W_{2n}'(z) \vert }{2n},$$
$$  \vert W_{2n}(z) \vert^2 \sum_{j=1}^{n} \left( \frac{1}{\vert z-z_j \vert^2}+ \frac{1}{\vert z-\overline{z_j} \vert^2}\right)\leq L(2n)^2,$$ with $\frac{z_j+\frac{1}{z_j}}{2}=x_j$ for $j=1, \cdots, n$ and
 for some  positive constants $B$ and $ L$,   $n$ large enough and  every $z \in \mathbb{T}$.

Let  $f$  be  a continuous function on $[-1,1]$  such that $\lambda(f,\delta)=o(\delta^{\frac{1}{2}})$, if $\delta \rightarrow 0$.

If  $l_{n-1}(f,x)$ is the Lagrange  interpolation polynomial such that $l_{n-1}(f,x_j)=f(x_j)$ for $j=1,\cdots,n,$
 then $l_{n-1}(f,x)$ converges to $f(x)$ uniformly on $[-1,1]$.
\label{Thm:4} \end{Thm}

\begin{proof}  It is easy to see that the polynomial $W_{2n}(z)$ has the following expression $W_{2n}(z)= \displaystyle \Pi_{j=1}^n(z-z_j)(z-\overline{z_j})$, with $\frac{z_j+\overline{z_j}}{2}=x_j$.\\
Let us define a continuous function on $\mathbb{T}$ by $F(z)=F(\overline{z})=f(x)$, with $x= \frac{z+ \frac{1}{z}}{2}$ and $z \in \mathbb{T}$. It is clear that
 $$\lambda(F,\delta)=\sup_{z_1,z_2 \in \mathbb{T};\vert z_1-z_2 \vert < \delta} \vert F(z_1)-F(z_2) \vert \leq \sup_{x_1,x_2 \in [-1,1];\vert x_1-x_2 \vert < \delta} \vert f(x_1)-f(x_2) \vert = \lambda(f,\delta).$$
If we take $W_{2n}(z)$ as nodal system on  $\mathbb{T}$, we can  consider the following Lagrange  interpolation problem: find the Laurent polynomial of Lagrange interpolation $L_{-n,n-1}(F;z) \in \Lambda_{-n,n-1}$ satisfying the interpolation conditions
\begin{equation*}
    L_{-n,n-1}(F;z_j)=L_{-n,n-1}(F;\overline{z_j})=f(x_j),\; j=1, \cdots, n,
\end{equation*}
 By applying  Theorem  \ref{TL:2} we have that $\lim_{n \rightarrow \infty} L_{-n,n-1}(F;z)=F(z)$ uniformly on $\mathbb{T}$.\\
On the other hand, for $x=\frac{z+\frac{1}{z}}{2}$ and $ z \in \mathbb{T}$ it holds
\begin{equation*}\begin{split}
    L_{-n,n-1}(F;z)=\sum_{j=1}^n \frac{W_{2n}(z) z_j^n}{z^n W'_{2n}(z_j)(z-z_j)} F(z_j)+\sum_{j=1}^n \frac{W_{2n}(z) \overline{z_j}^n}{z^n W'_{2n}(\overline{z_j})(z-\overline{z_j})} F(\overline{z_j})=\\
\sum_{j=1}^n \frac{p_n(x)}{p'_n(x_j)(x-x_j)}f(x_j)= l_{n-1}(f;x).
\end{split}\end{equation*}
Hence
 $\displaystyle \lim_{n\rightarrow \infty}l_{n-1}(f,x)=f(x)$
uniformly on $[-1,1]$.\\
\end{proof}
As a consequence we obtain, in the next corollary, a result that was proved by Szeg\H{o} in \cite{Sze} under weaker conditions. Although our result is not new,  we give the proof because the way in which it is obtained is different from Szeg\H{o}'s proof.

\begin{Cor}
Let  $f$  be   a continuous function on $[-1,1]$ such that $\lambda(f,\delta)=o(\delta^{ \frac{1}{2}})$, if $\delta \rightarrow 0$. Let $d \mu(x)=w(x)dx$ be a finite positive Borel measure on $[-1,1]$ satisfying the Szeg\H{o} condition $\int_{-1}^1 \frac{\log w(x)}{\sqrt{1-x^2}}dx > -\infty$,   and let $\{ P_n(x)\}$ be the MOPS($\mu$). Assume that
the function
$w(x)\sqrt{1-x^2}$ is positive on $[-1,1]$ and it is   analytic in an open set containing  $[-1,1]$.\\
 If $l_{n-1}(f,x)$ is the Lagrange interpolation polynomial satisfying the interpolation conditions $l_{n-1}(f,x_j)=f(x_j),\; j=1, \cdots,n$, where $\{x_j\}_{j=1}^n$ are the zeros of the orthogonal polynomial  $P_n(x)$, then
  \begin{equation*}
   \displaystyle \lim_{n\rightarrow \infty}l_{n-1}(f,x)=f(x)
\end{equation*}
uniformly on $[-1,1]$.
\label{Cor:3}\end{Cor}
\begin{proof} By using   the Szeg\H{o}  transformation, (see \cite{Sze}), the measure  $d \mu(x)$ becomes into the measure  $d \nu(\theta)=\frac{1}{2}w( \cos \theta) \vert \sin \theta \vert d \theta$,  which is in the  Szeg\H{o} class with Szeg\H{o} function having analytic extension up to $ \vert z \vert >1$, (see \cite{Nev}). If we denote by $\{ \phi_n(z)\}$ the MOPS($\nu$) and by  $\{ P_n(z)\}$ the MOPS($\mu$), then both sequences  are related by
\begin{equation*}
    P_n(x)= \frac{1}{2^n(1+ \phi_{2n}(0))}\frac{\phi_{2n}(z)+ \phi_{2n}^*(z)}{z^n}=\frac{1}{2^n(1+ \phi_{2n}(0))}\frac{\omega_{2n}(z,1)}{z^n}.
\end{equation*}
The zeros of $P_n(x)$, $x_1, \cdots,x_n,$ are simple and belongs to $(-1,1)$ and they are related with the zeros of $\omega_{2n}(z,1)$, $z_1,\cdots,z_n, z_{n+1}=\overline{z_n}, \cdots, z_{2n}=\overline{z_1},$ by $x_j= \frac{z_j+\overline{z_j}}{2},\;j=1,\cdots,n.$
By applying Theorem  \ref{LemL:1} we get that the system $\omega_{2n}(z,1)$ satisfies the hypothesis of Theorem \ref{Thm:4}. Then we have that $l_{n-1}(f,x)$ converges to $f(x)$ uniformly on $[-1,1]$.
\end{proof}

Analogous results can be obtained for other nodal systems related with those given in Corollary \ref{Cor:3}. Let $d \mu_1(x)=w(x)dx$ be a finite positive Borel measure on $[-1,1]$. Let us consider the measures
\begin{equation}d\mu_2(x)=(1-x^2)d \mu_1(x), \;d \mu_3(x)=(1-x) d \mu_1(x), \;d \mu_4(x)= (1+x) d \mu_1(x),
\label{eq:tri}\end{equation}
 and let us denote the MOPS with respect to these measures by $\{P_n(x, \mu_i)\}_{i=1}^4$. Let us consider the Szeg\H{o} transformed measure of $d \mu_1(x)$, $d \nu(\theta)= \frac{1}{2 }w(\cos \theta) \vert \sin \theta \vert d\theta$ with MOPS($\nu$), $\{\phi_n(z)\}$. Taking into account the relation between the measures, we can relate the  orthogonal sequences as follows, (see \cite{Gar})
\begin{equation*}
    \begin{split}
    P_n(x, \mu_1)= \frac{1}{2^n(1+ \phi_{2n}(0))z^n}w_{2n}(z,1),\\
   \sqrt{1-x^2} P_n(x, \mu_2)= \frac{1}{2^{n+1}\imath (1- \phi_{2n+2}(0))z^{n+1}}w_{2n+2}(z,-1),\\
   \sqrt{1-x} P_n(x, \mu_3)= \frac{1}{2^{n+\frac{1}{2}}\imath (1- \phi_{2n+1}(0))z^{n+\frac{1}{2}}}w_{2n+1}(z,-1),\\
   \sqrt{1+x} P_n(x, \mu_4)= \frac{1}{2^{n+\frac{1}{2}} (1+\phi_{2n+1}(0))z^{n+\frac{1}{2}}}w_{2n+1}(z,1).\\
\end{split}\end{equation*}
We denote by $\pm1, x_1,\cdots,x_n$ the zeros of $\sqrt{1-x^2} P_n(x, \mu_2)$, by $1, y_1,\cdots,y_n$ the zeros of $\sqrt{1-x} P_n(x, \mu_3)$ and by $-1, v_1,\cdots,v_n$ the zeros of $\sqrt{1+x} P_n(x, \mu_4)$.\\
 If we denote by $\pm 1,z_1,\cdots,z_n,\overline{z_1},\cdots, \overline{z_n}$ the zeros of $\omega_{2n+2}(z,-1)$, by $ 1,w_1,\cdots,w_n,\overline{w_1},\cdots, \overline{w_n}$ the zeros of $w_{2n+1}(z,-1)$, by $ -1,u_1,\cdots,u_n,\overline{u_1},\cdots, \overline{u_n}$ the zeros of $w_{2n+1}(z,1)$, then the following relations hold:
 $\Re (z_i)=x_i, \;\Re (w_i)=y_i,\; \Re (u_i)=v_i,\, i=1,\cdots,n.$ By taking nodal systems related with the zeros of $P_n(x, \mu_i),\;i=2,3,4$ we obtain the next result.

\begin{Thm} Let $f$ be a continuous function on $[-1,1]$ such that $\lambda(f,\delta)=o( \delta^{\frac{1}{2}})$, if $\delta \rightarrow 0$. Let $\mu_1$ be a finite positive Borel measure on $[-1,1]$, $d \mu_1(x)=w(x)dx$, satisfying the Szeg\H{o} condition. Assume that the function $w(x) \sqrt{1-x^2} $ is positive in $[-1,1]$ and it is analytic in an open set containing $[-1,1]$. Let  $d \mu_i(x), i=2,3,4,$ be the measures  given in (\ref{eq:tri}). Let us consider the Lagrange interpolation polynomials for the function $f$ with the following nodal systems:
\begin{enumerate}
\item [(i)] the zeros of $P_n(x, \mu_2)$ joint with $\pm 1$,
\item [(ii)] the zeros of $P_n(x, \mu_3)$ joint with $ 1$,
\item [(iii)] the zeros of $P_n(x, \mu_4)$ joint with $-1$.
\end{enumerate}
 Then the corresponding  Lagrange interpolation polynomials uniformly converge to $f(x)$ on $[-1,1]$.
\label{TL:3}\end{Thm}
\begin{proof} By the Szeg\H{o} transformation the measure $d \mu_1(x)$ becomes into the measure\\
 $d \nu(\theta)= \frac{1}{2 }w(\cos \theta) \vert \sin \theta \vert d\theta$, which is in the Szeg\H{o} class with Szeg\H{o} function having analytic extension up to $\vert z \vert >1$. We denote by $\{ \phi_n(z)\}$ the MOPS($\nu$). If we define a continuous function $F$ on  $\mathbb{T}$ by  $F(z)=F(\overline{z})=f(x)$ with $x=\frac{z+\frac{1}{z}}{2}$ and  $z \in \mathbb{T}$, then it is clear that $\lambda(F,\delta) \leq \lambda(f,\delta)$.

(i) We consider the para-orthogonal polynomial $\omega_{2n+2}(z,-1)$, whose zeros are $\pm 1, z_1,\cdots, z_n,$ $ \overline{z_1},\cdots, \overline{z_n} \in \mathbb{T}$ and they are related with the zeros of $P_n(x,\mu_2)$ by $x_j=\frac{z_j+ \overline{z_j}}{2};\;j=1,\cdots, n.$\\
    Let us consider the following Lagrange interpolation problem: find the Laurent polynomial of Lagrange interpolation $L_{-(n+1),n}(F;z)$ satisfying
\begin{equation*}\begin{split}
    L_{-(n+1),n}(F;z_j)=L_{-(n+1),n}(F;\overline{z_j})=F(z_j),\; j=1,\cdots,n,\\
     L_{-(n+1),n}(F;1)=F(1), \, L_{-(n+1),n}(F;-1)=F(-1).
\end{split}\end{equation*}
By applying Corollary   \ref{Cor:A1} we have that $\displaystyle  \lim_{n \rightarrow \infty}  L_{-(n+1),n}(F;z)=F(z)$ uniformly on $\mathbb{T}$.  If we take  $$ l_{n+1}(f,x)= \frac{L_{-(n+1),n}(F;z)+L_{-(n+1),n}(F;\frac{1}{z})}{2}$$ for
$x=\frac{z+\frac{1}{z}}{2}$, then  $ l_{n+1}(f,x)$ fulfills $ l_{n+1}(f,x_j)=f(x_j),\;  j=1,\cdots,n,$ and  $ l_{n+1}(f,\pm 1)=f(\pm 1)$. Therefore, $ l_{n+1}(f,x)$ is the Lagrange interpolation polynomial for the function $f$ and the nodal system given in (i) and  $\displaystyle  \lim_{n \rightarrow \infty}l_{n+1}(f,x)=f(x)$ uniformly on $[-1,1]$.

(ii) We consider the para-orthogonal polynomials $\omega_{2n+1}(z,-1)$ whose zeros, $1,w_1,\cdots,w_n,$  $\overline{w_1},\cdots,\overline{w_n}$, are related with the zeros of $P_n(x,\mu_3)$ by $y_j=\frac{w_j+\overline{w_j}}{2},\; j=1,\cdots,n.$\\
We pose  the problem of finding  the Laurent polynomial of Lagrange interpolation $L_{-n,n}(F;z)$ satisfying
\begin{equation*}\begin{split}
    L_{-n,n}(F;w_j)=L_{-n,n}(F;\overline{w_j})=F(w_j),\; j=1,\cdots,n,\\
     L_{-n,n}(F;1)=F(1).
\end{split}\end{equation*}
Since $\displaystyle  \lim_{n \rightarrow \infty}  L_{-n,n}(F;z)=F(z)$ uniformly on $\mathbb{T}$, if we define  $\displaystyle  l_{n}(f,x)= \frac{L_{-n,n}(F;z)+ L_{-n,n}(F;\frac{1}{z})}{2}$ for
$x=\frac{z+\frac{1}{z}}{2}$ and $ z \in \mathbb{T}$,  then  $ l_{n}(f,x)$ fulfills $ l_{n}(f,x_j)=f(x_j),\;  j=1,\cdots,n,$ and  $l_{n}(f, 1)=f(1)$. Therefore, $\displaystyle  \lim_{n \rightarrow \infty}l_{n}(f,x)=f(x)$ uniformly on $[-1,1]$.

(iii) It is obtained proceeding in the same way as in the previous items.
\end{proof}
\begin{Rem} \begin{enumerate}
\item[(i)] In particular, the  preceding result is valid for the following nodal systems: the zeros of the Tchebychef polynomials of the second kind joint with $\pm 1$, the zeros of the Tchebychef polynomials of the third  kind joint with $1$ and the zeros of the Tchebychef polynomials of the fourth kind joint with $-1$.
    \item [(ii)] Moreover it is also valid for the polynomial modifications, by positive polynomials,  of the  Bernstein measures corresponding to the Tchebychef measures mentioned before.
\end{enumerate}
\end{Rem}
\section{Trigonometric interpolation}
Next we obtain some consequences of Theorem \ref{TL:2}, which are related with the Lagrange trigonometric interpolation. Now the nodal points are in $[0,2 \pi]$ and they are obtained as follows. Let $d\mu(x)= w(x)dx$ be a positive finite Borel measure on $[-1,1]$ satisfying the Szeg\H{o} condition. Assume that the function $w(x) \sqrt{1-x^2}>0\; \forall x \in [-1,1]$ and it is analytic in an open set containing $[-1,1]$. If  $\{P_n(x)\}$ is  the MOPS($\mu$) and  $\{x_j\}_{j=1}^n$ are  the zeros of $P_n(x)$,  we consider the following nodal system on $[0,2 \pi],\, \{ \theta_j\}_{j=1}^{2n}$, such that $\theta_j= \arccos x_j,\, j=1, \cdots,n$ with $0< \theta_j < \pi$ and $\theta_{n+j}=2 \pi-\theta_{n-j+1}$ for $j=1, \cdots,n,$ that is, the points are symmetric with respect to $\pi$.
\begin{Thm} Let  $f$ be a real continuous function on $[0,2 \pi]$, with modulus of continuity $\lambda(f,\delta)=o(\delta^{\frac{1}{2}})$,  if $\delta \rightarrow 0$.

 Let the nodal system be
$\{\theta_j\}_{j=1}^{2n}$ with $\theta_j= \arccos x_j,\, j=1, \cdots,n$ with $0< \theta_j < \pi$ and $\theta_{n+j}=2 \pi-\theta_{n-j+1}$ for $j=1, \cdots,n,$ where $\{x_j\}_{j=1}^n$ are the zeros of the orthogonal polynomial $P_n(x)$ with respect to the measure $d \mu$. We also assume that $d\mu(x)= w(x)dx$ is a positive finite Borel measure on $[-1,1]$ satisfying the Szeg\H{o} condition and such that $w(x) \sqrt{1-x^2}>0\; \forall x \in [-1,1]$ and it is analytic in an open set containing $[-1,1]$.

 Then there is a Lagrange interpolation trigonometric polynomial of degree $ \leq n$, $\tau_n(\theta)$, such that $\tau_n(\theta_j)=f(\theta_j)$ for $j=1, \cdots, 2n$ and it satisfies that $\displaystyle  \lim_{n \rightarrow \infty} \tau_n(\theta)= f(\theta)$ uniformly on $[0,2 \pi]$.
\label{TL:4}\end{Thm}
\begin{proof}  Proceeding like in Corollary \ref{Cor:3} we obtain that the transformed measure of $d \mu(x)$ by  the Szeg\H{o} transformation,  $d \nu(\theta)$, satisfies the hypothesis of Theorem  \ref{LemL:1} and the para-orthogonal polynomials satisfy the bound condition of Theorem \ref{LemL:1}.

Let us define  $F$ by $F(e^{\imath \theta})=f(\theta)$ for $\theta \in [0, 2 \pi].$ Let $L_{-n,n-1}(F,z)$ be the Lagrange interpolation polynomial such that $L_{-n,n-1}(F,z_j)=f(\theta_j)$, where $z_j= e^{\imath \theta_j},\; j=1, \cdots,2n.$\\
Since $F$ is continuous on $\mathbb{T}$ and $\lambda(F,\delta)=o(\delta^{\frac{1}{2}})$, we can apply Corollary \ref{Cor:3}  and therefore\\
 $\displaystyle  \lim_{n \rightarrow \infty}  L_{-n,n-1}(F;z)=F(z)$ uniformly on $\mathbb{T}$.\\
If we take $\tau_n(\theta)=\Re (L_{-n,n-1}(F;e^{\imath \theta}))$ then $\tau_n(\theta)$ satisfies the interpolation conditions, $\tau_n(\theta_j)=f(\theta_j)$,  and $\displaystyle  \lim_{n \rightarrow \infty} \tau_n(\theta)=f(\theta)$.
\end{proof}
In the next result  we denote the integer part of $x$ by $[x]$ and we consider another type of nodal system on $[0,2 \pi].$

\begin{Thm} Let $\nu$ be a measure on  $[0,2 \pi]$ in the Szeg\H{o} class with Szeg\H{o} function having analytic extension up to $\vert z \vert >1$. Let $\{ z_j\}_{j=1}^n$ be the zeros of the para-orthogonal polynomials $\omega_n(z)=\phi_n(z)+ \tau \phi_n^*(z)$, with $\vert \tau \vert =1$, and  let $\theta_j \in [0, 2 \pi]$ such that $e^{\imath \theta_j}=z_j,\; j=1,\cdots, n$.

If $f$ is a continuous function on  $[0, 2 \pi]$ with
  $\lambda(f,\delta)=o(\delta^{\frac{1}{2}})$,  if $\delta \rightarrow 0$,
then    there is a  Lagrange interpolation trigonometric polynomial of  degree $ \leq [\frac{n}{2}]$, $\tau_{[\frac{n}{2}]}(\theta)$, such that $\tau_{[\frac{n}{2}]}(\theta_j)=f(\theta_j)$ for $j=1, \cdots, n$ and it satisfies that $\displaystyle  \lim_{n \rightarrow \infty} \tau_{[\frac{n}{2}]}(\theta)= f(\theta)$ uniformly on $[0,2 \pi]$.
\label{TL:5}\end{Thm}
\begin{proof} Let $F$ be a continuous function defined by $F(e^{\imath \theta})=f(\theta)$. Since $\lambda(F,\delta) \leq \lambda(f,\delta)$, then $\lambda(F,\delta)=o(\delta^{\frac{1}{2}})$.  By applying Corollary  \ref{Cor:A1} we obtain for $p(n)+q(n)=n-1$, with $\displaystyle \lim_{n \rightarrow \infty} \frac{p(n)}{n-1}=r$ and $ 0 < r < 1$, that there exists $L_{-p(n),q(n)}(F;z)$ such that $L_{-p(n),q(n)}(F;e^{\imath \theta_j})=f(\theta_j)$ and  $\displaystyle  \lim_{n \rightarrow \infty}  L_{-p(n),q(n)}(F;z)=F(z)$ uniformly on $\mathbb{T}$.\\
We distinguish the two following cases:\\
(i) If $n$ is even we take  $p(n)=\frac{n}{2}$ and $q(n)=\frac{n}{2}-1$. Then $[\frac{n}{2}]=\frac{n}{2}$.\\
(ii) If $n$ is odd we take  $p(n)=\frac{n-1}{2}$ and $q(n)=\frac{n-1}{2}$. Then $[\frac{n}{2}]=\frac{n-1}{2}$.\\
In any case the real part of  $ L_{-p(n),q(n)}(F;z)$ is a trigonometric polynomial of degree $[\frac{n}{2}]$, that satisfies the interpolation conditions and the convergence property.
\end{proof}

\end{document}